\documentclass[11pt]{article}
\usepackage{amsfonts}
\usepackage{amssymb}
\usepackage{amsthm,bm}
\usepackage{amsmath}

\baselineskip 5pt \lineskip 5pt \numberwithin{equation}{section}
\newtheorem{Prop}{Proposition}
\newtheorem{Thm}{Theorem}
\newtheorem{Def}{Definition}
\newtheorem{Lem}{Lemma}

\begin{document}
\title{Exact covering systems in number fields}
\author{Yupeng Jiang, Yingpu Deng\\\\ \quad Key Laboratory of Mathematics
Mechanization,
\\NCMIS, Academy of Mathematics and Systems Science,
\\Chinese Academy of Sciences, Beijing 100190, P.R. China\\ E-mail addresses: \{jiangyupeng,dengyp\}@amss.ac.cn}
\date{}
\maketitle

\begin{abstract}
It is well known that in an exact covering system in $\mathbb{Z}$,
the biggest modulus must be repeated. Very recently, Kim gave
an analogous result for certain quadratic fields, and Kim also conjectured that it must hold in any algebraic number field. In this paper, we prove Kim's conjecture. In other words, we prove that exact covering systems in any algebraic number
field must have repeated moduli.
\end{abstract}


\section{Introduction}

For $a, n\in\mathbb{Z}$ and
$n>1$, let $a\,\text{mod}\,n$ denote the set of
integers $\{a+kn\mid\ k\in\mathbb{Z}\}$. A finite collection of
congruence classes $\{a_i\,\text{mod}\,n_i{}\}_{i=1}^k$ is called
a covering system in $\mathbb{Z}$ if each integer belongs to at
least one congruence class. This concept was first introduced by
Erd\"os \cite{Erdos} in 1950, who constructed an infinite
arithmetic sequence of odd integers not representable as $2^n+p$
($p$ a prime) using a covering system. If moreover, each integer
belongs to exactly one congruence class, we say that it is an exact
covering system.

Kim generalized the definition to algebraic number fields \cite{Kim0, Kim1}.
Given a number field $K$, denote its ring of algebraic integers by
$\mathcal{O}_K$. For  $\alpha\in\mathcal{O}_K$ and a
nonzero ideal $I$ of $\mathcal{O}_K$ with $I\neq\mathcal{O}_K$, let
$\alpha+I$ denote the set of all the algebraic integers $\beta$ satisfying
$\beta\equiv\alpha\mod I$. $\{\alpha_i+I_i\}_{i=1}^k$ is called a covering
system in $K$ if each element in $\mathcal{O}_K$ belongs to at
least one congruence class. Of course we call it an exact covering
system if each element in $\mathcal{O}_K$ belongs to exactly one
congruence class.

Let $\{a_i\,\text{mod}\,n_i{}\}_{i=1}^k$ be  an exact covering
system in $\mathbb{Z}$. We assume $2\le n_1\le
n_2\le\dots\le n_k$. In 1971, M. Newman \cite{M.newman} proved that
$$n_{k-p(n_k)+1}=\dots=n_k,$$ where for an integer $n$,
$$p(n)=\min\{p\mid\ p|n,\ p\ \text{prime}\}.$$ In
\cite{nonanalytic}, M. A. Berger, A. Felzenbaum and A. S. Fraenkel
called a modulus $n_i$ division maximal if for $1\le j\le k$,
$n_i|n_j$ implies $n_i=n_j$. They proved that each division maximal $n_i$ must
be repeated at least $p(n_i)$ times. As the biggest modulus is division maximal,
this result generalizes Newman's. In \cite{improvements}, the
same authors got an improved result that each modulus $n_i$ must be repeated at least
$$\min\{G\left(\frac{n_i}{{\mbox {gcd}}(n_i,n_j)}\right)\mid\ n_j\neq n_i\}$$ times, where for an integer $n$,
$$G(n)=\max\{p^r\mid\ p^r|n,\ p\ \text{prime},r\geq0\}.$$ This result is better. For a
division maximal $n_i$, $n_j\neq n_i$ implies that ${\mbox {gcd}}(n_i,
n_j)$ is a proper divisor of $n_i$, so $G(n_i/{\mbox
{gcd}}(n_i,n_j))\ge p(n_i)$.

For exact covering systems in a number field $K$, very recently, in
\cite{Kim}, Kim asked whether there exists an exact covering system with distinct moduli in a number field and
he gave some nonexistence results when $K$ is a certain type of quadratic field with
additional constraints on the modulus ideals $I_i$. We are going to
prove the analogous results in \cite{nonanalytic} and
\cite{improvements} for exact covering systems in any number field
$K$.

This paper is organized as follows. In Section 2, we introduce
some notations about covering systems in number fields and
Kim's results \cite{Kim}. In Section 3, we'll introduce some
definitions from \cite{nonanalytic}, and then prove the analogous
result of \cite{nonanalytic} for number fields. We'll prove the
analogous result of \cite{improvements} for number fields in
Section 4. Both results greatly generalize the results of Kim in
\cite{Kim}.

\section{Notations and known results}

We always use the following notations and decompositions. Let $K$
be a number field, $\mathcal{O}_K$ its ring of integers and
$\{\alpha_i+I_i\}_{i=1}^k$ an exact covering system
of $\mathcal{O}_K$. Let $I=\cap_{i=1}^kI_i$ with prime ideal
decomposition
$$I=\prod_{j=1}^l\mathfrak{p}_j^{r_j},$$ where $\mathfrak{p}_j$'s are prime
$\mathcal{O}_K$-ideals and $r_j>0$. Each $I_i$ has decomposition
$$I_i=\prod_{j=1}^l\mathfrak{p}_j^{r_{i,j}}$$ with $0\le r_{i,j} \le
r_j$. For each $\mathfrak{p}_j$, assume that its norm satisfies
$N(\mathfrak{p}_j)=n_j$. Here the norm $N(I)$
of a nonzero ideal $I$ of $\mathcal{O}_K$ is the
cardinality of the finite ring $\mathcal{O}_K/I$.

Now we recall Kim's results in \cite{Kim}. The first two are
about imaginary quadratic fields.

\begin{Prop}
Let $K=\mathbb{Q}(\sqrt{-m})$ be an imaginary quadratic field and
$\{\alpha_i+I_i\}_{i=1}^k$ an exact covering system
of $\mathcal{O}_K$ with $N(I_1)\le N(I_2)\le\dots\le N(I_k)$. If
$I_k$ is principal, then it must be repeated.
\end{Prop}

\begin{Prop}
Let $K=\mathbb{Q}(\sqrt{-m})$ be an imaginary quadratic field with
class number two and $m\neq15,\,35,\,91,\,187,\,403$. If
$\{\alpha_i+I_i\}_{i=1}^k$ is an exact covering
system of $\mathcal{O}_K$, then the moduli can not be distinct.
\end{Prop}

Kim also got a result about real quadratic fields.

\begin{Prop}
Let $m$ be a positive integer. If
$\{\alpha_i+I_i\}_{i=1}^k$ is an exact covering
system of $\mathcal{O}_K$ with $K=\mathbb{Q}(\sqrt{m})$, where all
the moduli are principal, then any modulus of the largest norm
must be repeated.
\end{Prop}

\section{Exact covering systems in number fields}

We first give the definition of a division maximal ideal.
\begin{Def}
Let $\mathcal{I}=\{I_i\mid\ 1\le i\le k\}$ be a collection of
$\mathcal{O}_K$-ideals. An ideal $I\in \mathcal{I}$ is called
division maximal if $J\in \mathcal{I}$ and $I|J$ imply $I=J$.
\end{Def}

In this section, we will prove the following result.

\begin{Thm}
Let $\{\alpha_i+I_i\}_{i=1}^k$ be an exact covering system
of $\mathcal{O}_K$ for a number field $K$. If $I_i$ is
division maximal, then $I_i$ must
be repeated at least $\min\{n_j\mid\ \mathfrak{p}_j|I_i\}$ times.
\end{Thm}

It is obvious that the above theorem covers Kim's Proposition 2.
As an ideal with largest norm must be a division maximal ideal, our result
covers Propositions 1 and 3. Moreover we give the least repeated times and not just claim that there are repeated ideals.
We don't need any restriction of the number field  or the modulus ideals. This result holds for any number field.

To prove Theorem 1, we first recall the important concept
of parallelotope from \cite{nonanalytic}.

For $\bm{b} =(b_1, b_2,\dots, b_n) \in \mathbb{Z}^n$ with $b_i\ge
2,\ 1 \le i \le n$, define the lattice parallelotope or simply
parallelotope
\begin{align*}
P=P(n; \bm{b})&= \{\bm{c}=(c_1, c_2,\dots, c_n) \in
\mathbb{Z}^n\mid\ 0
\le c_i < b_i\ for\ 1 \le i \le n\}\\
&= B_1 \times B_2 \times\dots\times B_n,
\end{align*}
where $B_i = \{0, 1,\dots, b_i-1\}$ for $1 \le i \le n$. If
$b_1=b_2=\dots=b_n=b$, then $P(n;\bm{b})$ is also called the cube
$U(n;b)$.

\begin{Def}
Given a parallelotope $P=P(n; \bm{b})$, let $I \subseteq \{1,2,
\dots,n\}$. An $I$-cell or simply cell $C$ of $P$ is a set of the
form
\begin{align*}
C &= \{\bm{s}=(s_1,s_2, \dots, s_n) \in \mathbb{Z}^n\mid\ 0
\le s_i < b_i\ for \ i \in I, s_i=u_i\ for\ i \notin I\}\\
&= D_1 \times D_2 \times\dots\times D_n,
\end{align*}
where $D_i=\{0,1,\dots, b_i-1\}$ for $i \in I$, $D_i=\{u_i\}$ for $i
\notin I$. Here $\bm{u}=(u_1,u_2,\dots,u_n)$ is an arbitrary point in $P$. The set $I$ is called the index of $C$, and denoted by
$I=I(C)$.
\end{Def}
\begin{Def}
A partition $\tau$ of a parallelotope $P$ into cells is called a
cell partition of $P$. A cell $C \in \tau$ is said to be subset
minimal if $C' \in \tau$ and $I(C') \subseteq I(C)$ imply
$I(C')=I(C)$.
\end{Def}
We also need the following lemma from \cite{nonanalytic}.

\begin{Lem}
Let $\tau$ be a cell partition of $P(n; \bm{b})$ into at least two
cells and let $E \in \tau$ be subset minimal. Put $b=\min\{b_i\
\mid\ i\notin I(E)\}$. Then $\tau$ contains at least $b$
$I(E)$-cells.
\end{Lem}
Now we consider exact covering systems in number fields. For an exact covering system
$\{\alpha_i+I_i\}_{i=1}^k$, let $I=\cap_{i=1}^kI_i$ with prime ideal decomposition $I=\prod_{j=1}^l\mathfrak{p}_j^{r_j}$. The main
idea is giving a map that sends each residue class $\alpha+I$ to a point
of certain parallelotope. This map sends $\alpha_i+I_i$ to a cell of the parallelotope, and a
division maximal ideal corresponds to a subset minimal cell. In Berger et al.'s article,
they used the fact that $\mathbb{Z}/n\mathbb{Z}$ is cyclic and
defined an addition operation on the corresponding parallelotope.
But for an ideal $I$, $\mathcal{O}_K/I$ is not always cyclic.
Fortunately, we can use the Chinese remainder theorem to get what we
need.

Let $B_j=\{\beta_{j,1},\dots,\beta_{j,n_j}\}$ be a complete set of
representatives of $\mathcal{O}_K$ modulo $\mathfrak{p}_j$. We
can define a bijective map $f_j$ by
\begin{eqnarray*}
f_j: B_j&\rightarrow&\{0,1,\dots,n_j-1\}\\
\beta_{j,i}&\mapsto& i-1.
\end{eqnarray*}
Fix some $t_j \in \mathfrak{p}_j\backslash \mathfrak{p}_j^2$,
then for any positive integer $r$,
$$B_j^r:=\{\sum_{k=0}^{r-1}\gamma_k\,t_j^k\mid\ \gamma_k \in
B_j\}$$ is a complete set of representatives of
$\mathcal{O}_K$ modulo $\mathfrak{p}_j^r$. This is easy to prove by using valuations as follows. If $\alpha,\beta\in B_j^r$ with $\alpha=\sum_{k=0}^{r-1}a_k\,t_j^k$,  $\beta=\sum_{k=0}^{r-1}b_k\,t_j^k$ and $\alpha\ne \beta$, then the valuation $v_{\mathfrak{p}_j}(\alpha-\beta)=\min\{i\ |\ a_i\ne b_i\}$ is less than $r$. But if $\alpha\equiv \beta\mod \mathfrak{p}_j^r$, we have $v_{\mathfrak{p}_j}(\alpha-\beta)\ge r$. So we have $\alpha\not\equiv \beta\mod \mathfrak{p}_j^r$ if $\alpha\ne\beta$. $B_j^r$ has $n_j^r=N(\mathfrak{p}_j^r)$ elements so it is a complete set of representatives. For the definition and properties of valuations see \cite{Flo}.
We naturally extend
$f_j$ to a map
\begin{eqnarray*}
f_j^r: B_j^r &\rightarrow&
U(r;n_j)\\
\sum_{k=0}^{r-1}\gamma_k\,t_j^k &\mapsto&
(f_j(\gamma_0),\dots,f_j(\gamma_{r-1})).
\end{eqnarray*}
Of course $f_j^r$ is also a bijective map.

Now we can define a map $f$ that sends each $\alpha+I$ to a
point in $P(n; \bm{b})$, here $n=\sum_{j=1}^lr_j$ and
$\bm{b}=(\underbrace{n_1,\dots,n_1}_{r_1},\underbrace{n_2,\dots,n_2}_{r_2},
\dots,\underbrace{n_l,\dots,n_l}_{r_l})$, i.e.$$P(n;
\bm{b})=\underbrace {f_1(B_1)\times\dots\times
f_1(B_1)}_{r_1}\times\dots\times\underbrace{f_l(B_l)\times\dots\times
f_l(B_l)}_{r_l}.$$
Assume for each $1\le j \le l$,
$$\alpha \equiv \sum_{k=0}^{r_j-1}\gamma_{j,k}\,t_j^k \mod
\mathfrak{p}_j^{r_j},$$ then define
$$f(\alpha+I)=(f_1^{r_1}(\sum_{k=0}^{r_1-1}\gamma_{1,k}\,t_1^k),\dots,
f_l^{r_l}(\sum_{k=0}^{r_l-1}\gamma_{l,k}\,t_l^k)).$$

\begin{Lem}
The map $f: \mathcal{O}_K/I\longrightarrow P(n;\bm{b})$ is
bijective. For all $\alpha_i \in \mathcal{O}_K$, we have that $f(\alpha_i)$ is a cell of $P(n;\bm{b})$ and division maximal ideals correspond to
subset minimal cells. Moreover if
$\{\alpha_i+I_i\}_{i=1}^k$ is an exact covering system of
$\mathcal{O}_K$, then $\{f(\alpha_i+I_i)\mid\
i=1,2,\dots, k\}$ is a cell partition of $P(n;\bm{b})$.
\end{Lem}

\begin{proof}
By the Chinese remainder theorem, each $\alpha+I$
corresponds one-to-one to, for $1\le j\le l$, $$\alpha \equiv
\sum_{k=0}^{r_j-1}\gamma_{j,k}\,t_j^k \mod
\mathfrak{p}_j^{r_j}\mbox{ with }\sum_{k=0}^{r_j-1}\gamma_{j,k}\,t_j^k
\in B_j^{r_j}.$$ Hence $f$ is well-defined. If $f(\alpha+I)=f(\beta+I)$, then for all $1\le j\le l$,
$$\alpha\equiv\beta\mod\mathfrak{p}_j^{r_j},$$ so $\alpha\equiv\beta\mod I$. Thus $f$ is injective. This is a map between finite sets with the same cardinality, so $f$ is also surjective. We now prove
$f(\alpha_i+I_i)$ is a cell of $P(n;\bm{b})$. Without loss of generality,
we prove it for the case $i=1$. According to the definition of $B_j^r$,
$$S_j=\{\sum_{k=r_{1,j}}^{r_j-1}\gamma_k\,t_j^k|\ \gamma_k \in
B_j\}$$ is a complete set of representatives of
$\mathfrak{p}_j^{r_{1,j}}$ modulo $\mathfrak{p}_j^{r_j}$, for $1\le j
\le l$. If $$\alpha_1 \equiv
\sum_{k=0}^{r_{1,j}-1}a_{j,k}\,t_j^k\mod
\mathfrak{p}_j^{r_{1,j}}\mbox{ with } a_{j,k}\in B_j,$$
then $\alpha_1+I_1$ is just the union of all the residue classes $\beta+I $ with
$$\beta \equiv \sum_{k=0}^{r_{1,j}-1}a_{j,k}\,t_j^k+s_j \mod
\mathfrak{p}_j^{r_j}\mbox{ with } s_j \in S_j\mbox{ and } 1\leq j\leq l.$$ So
$f(\alpha_1+I_1)=C_{1,0}\times\dots\times
C_{1,r_1-1}\times\dots\times C_{l,0}\times\dots\times
C_{l,r_l-1}$, where
\begin{align*}
C_{j,k}&=\{f_j(a_{j,k})\}\qquad\text{for}\ 0\le k \le r_{1,j}-1,\\
C_{j,k}&=f_j(B_j)\hspace{13mm}\text{for}\ r_{1,j}\le k \le r_j-1.
\end{align*}
Therefore $f(\alpha_1+I_1)$ is a cell of $P(n;\bm{b})$.

According to the above discussion, for
$I_i=\prod_{j=1}^l\mathfrak{p}_j^{r_{i,j}}$,
$f(\alpha_i+I_i)$ has index
$$I(f(\alpha_i+I_i))=\cup_{j=1}^l(\{r_{i,j}+1,\dots,r_j\}+\sum_{k=1}^{j-1}r_k).$$
Here for a set $S$, $S+n=\{s+n\ |\ s\in S\}$. As
\begin{align*}
I_i|I_t&\Leftrightarrow r_{i,j}\le r_{t,j},\ 1\le j\le l \\
&\Leftrightarrow\cup_{j=1}^l(\{r_{i,j}+1,\dots,r_j\}+\sum_{k=1}^{j-1}r_k)
\supseteq \cup_{j=1}^l(\{r_{t,j}+1,\dots,r_j\}+\sum_{k=1}^{j-1}r_k)\\
&\Leftrightarrow I(f(\alpha_i+I_i))\supseteq
I(f(\alpha_t+I_t)),
\end{align*}
$I_i$ is division maximal if and only if
$f(\alpha_i+I_i)$ is subset minimal. The last claim is obvious.
\end{proof}

Now we can prove Theorem 1.

\textit{Proof of Theorem 1}.\quad By Lemma 2, $\{f(\alpha_i+I_i)\ |\ 1\le i\le k\}$ is a cell partition of $P(n;\bm{b})$. For a division maximal ideal $I_i$,
$f(\alpha_i+I_i)$ is subset minimal. By Lemma 1 there are at least $$\min\{b_t\ |\ t \notin I(f(\alpha_i+I_i))\}$$ cells with index $I(f(\alpha_i+I_i))$. According to
Lemma 2, for $I_i=\prod_{j=1}^l\mathfrak{p}_j^{r_{i,j}}$, the index of $f(\alpha_i+I_i)$ is $$\cup_{j=1}^l(\{r_{i,j}+1,\dots,r_j\}+\sum_{k=1}^{j-1}r_k).$$
Two cells with the same index correspond to the same ideal. $I_i$ must be repeated at least
$$\min\{b_t\ |\ t \notin I(f(\alpha_i+I_i))\}$$ times. We have
\begin{align*}
t \notin I(f(\alpha_i+I_i))&\Longleftrightarrow t\notin \cup_{j=1}^l(\{r_{i,j}+1,\dots,r_j\}+\sum_{k=1}^{j-1}r_k)\\
&\Longleftrightarrow t \in \cup_{j=1}^l(\{1,\dots,r_{i,j}\}+\sum_{k=1}^{j-1}r_k).
\end{align*}
If $r_{i,j}=0$, then
$\{1,\dots,r_{i,j}\}+\sum_{k=1}^{j-1}r_k=\emptyset$. As $r_{i,j}\ge 1$ if and only if $\mathfrak{p}_j|I_i$, then $I_i$
must be repeated at least $\min\{n_j\mid\ \mathfrak{p}_j|I_i\}$ times. The proof is complete. \qed

\section{An improvement of the above result}

In this section, we'll prove an improvement of the above result.
\begin{Thm}
If $\{\alpha_i+I_i\}_{i=1}^k$ is an exact covering
system of $\mathcal{O}_K$ with not all the ideals the same, then
each $I_i$ must be repeated at least
$$\min\{G(\frac{I_i}{I_i+I_j})|\ \ {I_j\ne
I_i}\}$$ times, where for an ideal $I$, $G(I)=
\max\{N(\mathfrak{p}^r)\mid\ \mathfrak{p}^r|I,\ \mathfrak{p}\
\text{prime ideal},r\ge 0\}$. Here $\frac{I_i}{I_i+I_j}$ is a division in the group of fractional ideals, not some kind of quotient.
\end{Thm}
It is easy to see that this bound is better than the bound in Theorem 1. For an
ideal $I_i$, if $I_i$ is not division maximal, then there exists
$I_j\neq I_i$ and $I_i|I_j$, so $I_i+I_j=I_i$ and
$G(I_i/(I_i+I_j))=G(\mathcal{O}_K)=1$. This bound is trivial. If
$I_i$ is division maximal, then for each $I_j$ with $I_j\ne I_i$,
we have $ I_i+I_j\supsetneq I_i$. Thus $G(I_i/(I_i+I_j))\ge \min\{n_j\mid\ \mathfrak{p}_j|I_i\}$
and the bound is better.

We need a new map to prove this result. Let $B_j$, $f_j$, $B_j^r$ be
just as in the last section. Now we define a map
\begin{eqnarray*}
\bar{f}_j^r:\quad B_j^r &\rightarrow&
\{0,1,\dots,n_j^r-1\}\\
\sum_{k=0}^{r-1}\gamma_k\,t_j^k &\mapsto&
\sum_{k=0}^{r-1}f_j(\gamma_k)\,n_j^{r-1-k}.
\end{eqnarray*}
The map $\bar{f}_j^r$ sends an element of $B_j^r$ to an integer represented in base $n_j$. The
coefficient of $t_j^k$ corresponds to the coefficient of
$n_j^{r-1-k}$. Obviously $\bar{f}_j^r$ is bijective. Now we can
define a map $\bar{f}$ that sends each $\alpha+I$ to a point
in $P(l; \bm{d})$, here $\bm{d}=(n_1^{r_1}, n_2^{r_2},\dots,
n_l^{r_l})$. Assume for each $1\le j \le l$,
$$\alpha \equiv \sum_{k=0}^{r_j-1}\gamma_{j,k}\,t_j^k \mod
\mathfrak{p}_j^{r_j},$$ then define
$$\bar{f}(\alpha+I)=(\bar{f}_1^{r_1}(\sum_{k=0}^{r_1-1}\gamma_{1,k}\,t_1^k),\dots,
\bar{f}_l^{r_l}(\sum_{k=0}^{r_l-1}\gamma_{l,k}\,t_l^k)).$$
Obviously $\bar{f}$ is bijective too. It maps each $\alpha_i+I_i$
to a subset of $P(l;\bm{d})$. Taking
$\alpha_1+I_1$ for example, just as in the above,
$$S_j=\{\sum_{k=r_{1,j}}^{r_j-1}\gamma_k\,t_j^k|\ \gamma_k \in
B_j\}$$ is a complete set of representatives of
$\mathfrak{p}_j^{r_{1,j}}$ modulo $\mathfrak{p}_j^{r_j}$, for
$1\le j \le l$. If
$$\alpha_1 \equiv \sum_{k=0}^{r_{1,j}-1}a_{j,k}\,t_j^k\mod \mathfrak{p}_j^{r_{1,j}},$$
then $\alpha_1+I_1$ is just the union of all the residue classes $\beta+I $ with
$$\beta \equiv \sum_{k=0}^{r_{1,j}-1}a_{j,k}\,t_j^k+s_j \mod \mathfrak{p}_j^{r_j}\mbox{ with } s_j \in S_j\mbox{ and } 1\leq j\leq l.$$
As each $\gamma_k$ can be an arbitrary element of $B_j$, we have
$$\bar{f}(\alpha_1+I_1)=C_1\times C_2\times \dots\times C_l,$$
where, according to the definition of $\bar{f}_j^{r_j}$,
$$C_j=\{c_j, c_j+1,\dots,c_j+n_j^{r_j-r_{1,j}}-1\}\subseteq\{0,1,\dots,n_j^{r_j}-1\}$$
and $c_j$ is some nonnegative integer such that
$n_j^{r_j-r_{1,j}}|c_j$. This is the key point in the following proof.

Now we begin to prove the above theorem. We first prove the
special case that there exists some $I_i$ such that $I_j|I_i$ for
all $j$. Without loss of generality, we may assume $I_1$ is such
an ideal, i.e. $r_{1,j}=r_{j}$, for each $1\le j \le l$. In the following proof, we use
$\bar{f}(\alpha_1+I_1)$ as the image of $\bar{f}$ at $\alpha_1+I_1$, and also use it as a singleton. We hope this doesn't cause confusion.

\begin{Lem}
Let $\{\alpha_i+I_i\}_{i=1}^k$ be an exact covering system
with $I_j|I_1$ for all $1\leq j\leq k$, and not all the ideals the same. Then
$I_1$ must be repeated at least
$$\min\{G(I_1/I_j)\mid\ \ {I_j\ne I_1}\}$$ times.
\end{Lem}

\begin{proof}
It is obvious that $\{\bar{f}(\alpha_i+I_i)|1\le i \le k\}$ is a
partition of $P(l;\bm{d})$, here $\bm{d}$ and $\bar{f}(\alpha_i+I_i)$ are defined as above and
$\bar{f}(\alpha_1+I_1)$ is a singleton. Denote $x=\min\{G(I_1/I_j)\mid\ {I_j\ne I_1}\}$. We are going to prove that in such a partition there are at least $x$ singletons, since each singleton corresponds to an ideal equal to $I_1$, hence $I_1$ must be repeated at least $x$ times.

First we consider the case when the point $\bar{f}(\alpha_1+I_1)\in U(l; x)$. According to the definition of $x$, the cardinality of $U(l;x)\cap P(l;\bm{d})$ is a multiple of $x$. We have that $$\{\bar{f}(\alpha_i+I_i)\cap U(l;x)\ |\ 1\le i \le k\}$$ is a partition of $U(l;x)\cap P(l;\bm{d})$ and
$$\bar{f}(\alpha_i+I_i)=C_1\times C_2\times \dots\times C_l,$$ where
$$C_j=\{c_j, c_j+1,\dots,c_j+n_j^{r_j-r_{i,j}}-1\}\subseteq\{0,1,\dots,n_j^{r_j}-1\}$$ and $n_j^{r_j-r_{i,j}}|c_j$.
If $I_i=I_1$, then $\bar{f}(\alpha_i+I_i)\cap U(l;x)$
is either a singleton or empty. There is nothing to prove. For $I_i\neq I_1$, assume $\bar{f}(\alpha_i+I_i)\cap U(l;x)$ is not empty, then we have
$C_j\cap\{0,1,\dots,x-1\}\ne \emptyset$ for all $1\leq j\leq l$. Hence $c_j<x$. Let $G(I_1/I_i)=n_j^{r_j-r_{i,j}}$ for some $j$, then
$x\le n_j^{r_j-r_{i,j}}$. We have that $c_j=0$ for this $j$, and
the cardinality of $\bar{f}(\alpha_i+I_i)\cap U(l;x)$ is a multiple of $x$. We conclude that all the nonempty
$\bar{f}(\alpha_i+I_i)\cap U(l;x)$ are singletons or have cardinality divisible by $x$ and if $I_i\neq I_1$, $\bar{f}(\alpha_i+I_i)\cap U(l;x)$ can not be a singleton. As there is one singleton $\bar{f}(\alpha_1+I_1)$, there are at least $x$ singletons. Therefore $I_1$ must be repeated at least $x$ times.

For the case $\bar{f}(\alpha_1+I_1) \notin U(l;x)$, we translate $U$ by a certain vector $\bm{v}=(v_1,\dots,v_l)$ and denote it by $U'$.
Let $\bar{f}(\alpha_1+I_1)=(a_1,\dots,a_l)\in P(l;\bm{d})$. We are going to define each $v_j$ as follows. For $1\leq j\leq l$, we distinguish two cases. For the $j-$th axis, if $G(I_1/I_m)=n_j^e$ for some $m$ with $I_m\ne I_1$, we let $e_j$ be the smallest $e$ with $1<G(I_1/I_m)=n_j^e$. There exists a unique nonnegative integer $N_j$ such that
$N_j\cdot n_j^{e_j}\le a_j< (N_j+1)\cdot n_j^{e_j}$. As $x\le n_j^{e_j}$, we can choose a nonnegative integer $v_j$ such that
$$a_j \in \{v_j,v_j+1,\dots,v_j+x-1\}\subseteq \{N_j\cdot n_j^{e_j},N_j\cdot n_j^{e_j}+1,\dots,(N_j+1)\cdot n_j^{e_j}-1\}.$$
For the $j-$th axis, if $G(I_1/I_m)\ne n_j^e$ for all $m$ with $I_m\ne I_1$ and $e>0$, we can choose a nonnegative integer $v_j$ such that $$a_j \in \{v_j,v_j+1,\dots,v_j+x-1\}.$$

It is obvious that $\bar{f}(\alpha_1+I_1)\in U'$. Now we want to prove that the cardinality of $U'\cap P(l;\bm{d})$ is a multiple of $x$.

According to the definitions of $x$ and $e_j$'s, $x=n_{j'}^{e_{j'}}$ for some $j'$ with $1\le j'\le l$. As
$N_{j'}\cdot n_{j'}^{e_{j'}}\le a_{j'}\le n_{j'}^{r_{j'}}-1$, then $N_{j'}\cdot n_{j'}^{e_{j'}}<n_{j'}^{r_{j'}}$. Both sides are multiples of $n_{j'}^{e_{j'}}$, so we have $n_{j'}^{r_{j'}}\ge (N_{j'}+1)\cdot n_{j'}^{e_{j'}}$, which means
$$\{N_{j'}\cdot n_{j'}^{e_{j'}},N_{j'}\cdot n_{j'}^{e_{j'}}+1,\dots,(N_{j'}+1)\cdot n_{j'}^{e_{j'}}-1\}\subseteq \{0,1,2,\dots, n_{j'}^{r_{j'}}-1\}.$$
Hence the cardinality of $U'\cap P(l;\bm{d})$ is a multiple of $x$.

Now we consider the intersections $\bar{f}(\alpha_i+I_i)\cap U'$ with $1\leq i\leq k$. We want to prove that these sets are either empty, or singletons or have a cardinality divisible by $x$.

If $I_i=I_1$, $\bar{f}(\alpha_i+I_i)\cap U'$ is obviously either a singleton or empty. For $I_i\neq I_1$, $G(I_1/I_i)=n_j^{r_j-r_{i,j}}$ for some $j$, and for this $j$ we have defined $v_j$ and $N_j$. Let $C_j$ be defined as above. If $\bar{f}(\alpha_i+I_i)\cap U'$ is not empty, $C_j\cap \{v_j,v_j+1,\dots,v_j+x-1\}\ne \emptyset$, and hence
$$C_j\cap \{N_j\cdot n_j^{e_j},N_j\cdot n_j^{e_j}+1,\dots,(N_j+1)\cdot n_j^{e_j}-1\} \ne \emptyset.$$
According to the definition of $e_j$, we have $e_j\le r_j-r_{i,j}$ and then $n_j^{e_j}|c_j$. We claim that $c_j\le N_j\cdot n_j^{e_j}$. If not, as $n_j^{e_j}|c_j$,
then $c_j\ge(N_j+1)\cdot n_j^{e_j}$, a contradiction. For the same reason
$c_j+n_j^{r_j-r_{1,j}}-1\ge (N_j+1)\cdot n_j^{e_j}-1$, so
$$\{v_j,v_j+1,\dots,v_j+x-1\}\subseteq \{N_j\cdot n_j^{e_j},N_j\cdot n_j^{e_j}+1,\dots,(N_j+1)\cdot n_j^{e_j}-1\}\subseteq C_j.$$
Then for $I_i\ne I_1$, $\bar{f}(\alpha_i+I_i)\cap U'$ is either empty or has cardinality divisible by $x$.
As in the first case, there are at least $x$ singletons and $I_1$ must be repeated at least $x$ times.
\end{proof}

Let $\{\alpha_i+I_i\}_{i=1}^k$ be any exact covering
system in $K$. We are going to get an exact covering system with the
above property from it. As $B_j$ is a complete set of representatives of
$\mathcal{O}_K$ modulo $\mathfrak{p}_j$, we let the
representative of $0$ modulo $\mathfrak{p}_j$ be just $0$, i.e.
$0\in B_j$ for all $1\le j\le l$. Assume $S$ is a complete set of
representatives of $\mathcal{O}_K$ modulo $I$. If $I'$ is any
ideal with $I'|I$, then $I'$ has the following decomposition
$$I'=\prod_{j=1}^l\mathfrak{p}_j^{s_j}$$ with $0\le s_j\le r_j$. We
define a subset $S_{I'}$ of $S$ by $$S_{I'}=\{\alpha\in S\mid\
\alpha\equiv\sum_{k=0}^{s_j-1}a_{j,k}t_j^{k}\mod
\mathfrak{p}_j^{r_j},\,\text{for all}\ 1\le j\le l\ \}.$$ The set $S_{I'}$
includes these elements of $S$
having the last $r_j-s_j$ coefficients 0 modulo $\mathfrak{p}_j^{r_j}$. It is easy to see that
$S_{I'}$ is a complete set of representatives of
$\mathcal{O}_K$ modulo $I'$.

\begin{Lem}
For each $(\alpha_i+I_i)\cap S_{I'}\neq\emptyset$, we
have
$$\alpha_i+I_i\subseteq(\alpha_i+I_i)\cap
S_{I'}+I',$$ which means that each element of
$\alpha_i+I_i$ can be written as a sum of an element
of $(\alpha_i+I_i)\cap S_{I'}$ and an element of
$I'$.
\end{Lem}

\begin{proof}
As $(\alpha_i+I_i)\cap S_{I'}\neq\emptyset$, choose
$\alpha\in(\alpha_i+I_i)\cap S_{I'}$, then
$$\alpha\equiv\sum_{k=0}^{s_j-1}a_{j,k}\,t_j^{k}\mod
\mathfrak{p}_j^{r_j}\mbox{ with } a_{j,k}\in B_j.$$ Every element in
$\alpha_i+I_i$ is of the form $\alpha+\beta$ with
$\beta\in I_i$, and assume
$$\beta\equiv\sum_{k=r_{i,j}}^{r_j-1}b_{j,k}\,t_j^{k}\mod
\mathfrak{p}_j^{r_j}\mbox{ with } b_{j,k}\in B_j,$$
$$\alpha+\beta\equiv\sum_{k=0}^{r_j-1}c_{j,k}\,t_j^{k}\mod
\mathfrak{p}_j^{r_j}\mbox{ with } c_{j,k}\in B_j.$$ Choose the element
$\alpha'\in S_{I'}$ with
$$\alpha'\equiv\sum_{k=0}^{s_j-1}c_{j,k}\,t_j^{k}\mod
\mathfrak{p}_j^{r_j},$$ and let $\gamma=\alpha+\beta-\alpha'$, then
$$\gamma\equiv\sum_{k=s_j}^{r_j-1}c_{j,k}\,t_j^{k}\mod
\mathfrak{p}_j^{r_j},$$ so $\gamma\in I'$. As
$\alpha+\beta=\alpha'+\gamma$, we only need to prove
$\alpha'\equiv\alpha\mod I_i$. We have
$$\sum_{k=0}^{s_j-1}a_{j,k}t_j^{k}+\sum_{k=r_{i,j}}^{r_j-1}b_{j,k}\,t_j^{k}\equiv
\sum_{k=0}^{s_j-1}c_{j,k}\,t_j^{k}+\sum_{k=s_j}^{r_j-1}c_{j,k}\,t_j^{k}\mod
\mathfrak{p}_j^{r_j},$$ for $1\le j\le l$. If $s_j\ge r_{i,j}$,
the last item of each side in the above equation belongs to
$\mathfrak{p}_j^{r_{i,j}}$, so
$$\sum_{k=0}^{s_j-1}a_{j,k}t_j^{k}\equiv\sum_{k=0}^{s_j-1}c_{j,k}\,t_j^{k}\mod
\mathfrak{p}_j^{r_{i,j}}.$$ If $s_j<r_{i,j}$, we just have
\begin{align*}
c_{j,k}&=a_{j,k}\hspace{10mm}\text{for}\ 0\le k\le s_j-1,\\
c_{j,k}&=0\hspace{14mm}\text{for}\ s_j\le k\le r_{i,j}-1,\\
c_{j,k}&=b_{j,k}\hspace{10mm}\text{for}\ r_{i,j}\le k\le r_j-1.
\end{align*}
Hence we always have $\alpha'\equiv\alpha\mod
\mathfrak{p}_j^{r_{i,j}}$ for all $1\le j\le l$, then
$\alpha'\equiv\alpha\mod I_i$. The proof is complete.
\end{proof}

\begin{Lem}
Let $\{\alpha_i+I_i\}_{i=1}^k$ be an exact covering
system in $K$. For any fixed $j$ with $1\leq j\leq k$, let $S_{I_j}$ be defined
as above and $J=\{i\mid\ (\alpha_i+I_i)\cap
S_{I_j}\neq\emptyset\}$. Then
$$\{\alpha_i+(I_i+I_j)\mid\ i\in J\}$$ is an exact
covering system in $K$.
\end{Lem}

\begin{proof}
As $\{\alpha_i+I_i\}_{i=1}^k$ is an exact covering
system in $K$, we have
\begin{align*}
S_{I_j}&\subseteq\bigcup_{i=1}^k(\alpha_i+I_i)\cap S_{I_j}\\
&=\bigcup_{i\in J}(\alpha_i+I_i)\cap S_{I_j}\\
&\subseteq\bigcup_{i\in J}\alpha_i+I_i\\
&\subseteq\bigcup_{i\in J}\alpha_i+(I_i+I_j).
\end{align*}
Since $S_{I_j}$ is a complete set of representatives of
$\mathcal{O}_K$ modulo $I_j$ and each $(I_i+I_j)|I_j$, so
$\{\alpha_i+(I_i+I_j)\mid\ i\in J\}$ is a covering
system in $K$.

If $x\in\alpha_i+(I_i+I_j)$, then
$x=\alpha_i+\beta_1+\gamma_1$ with $\beta_1\in I_i$ and $\gamma_1\in I_j$.
From the above Lemma 4, we have $\alpha_i+\beta_1=\alpha'+\delta$ with
$\alpha'\in (\alpha_i+I_i)\cap S_{I_j}$ and $\delta\in
I_j$. So we may assume $$x=\beta+\gamma\mbox{ with }\beta
\in(\alpha_i+I_i)\cap S_{I_j}\mbox{ and }\gamma\in I_j.$$ Moreover
if $x\in\alpha_{i'}+(I_{i'}+I_j)$, then
$$x=\beta'+\gamma'\mbox{ with }\beta'\in(\alpha_{i'}+I_{i'})\cap
S_{I_j}\mbox{ and }\gamma'\in I_j.$$ We have $\beta\equiv\beta'\mod I_j$, but
different elements of $S_{I_j}$ are in different residue classes modulo $I_j$, so
$\beta=\beta'$. Then $\beta\in\alpha_i+I_i$ and also
$\beta\in\alpha_{i'}+I_{i'}$. As
$\{\alpha_i+I_i\}_{i=1}^k$ is an exact covering
system in $K$, we have $i=i'$. Hence
$\{\alpha_i+(I_i+I_j)\mid\ i\in J\}$ is an exact
covering system in $K$.
\end{proof}

\textit{Proof of Theorem 2}. \quad Let
$\{\alpha_i+I_i\}_{i=1}^k$ be an exact covering
system with not all the ideals the same. For an arbitrary ideal $I_j$ in the covering system, if $I_j$ is not division maximal, as we have seen before, the bound in Theorem 2 is trivial. So we can assume that $I_j$ is division maximal.

For a division maximal ideal $I_j$, let $J$ be defined as in Lemma 5.
By Lemma 5, $\{\alpha_i+(I_i+I_j)\mid\ i\in J\}$ is an exact
covering system in which each ideal $I_i+I_j$ is a divisor of $I_j$. It is easy to see we have that $j\in J$
and $I_j+I_j=I_j$. We are going to prove that not all the ideals $I_i+I_j$ for $i\in J$ are equal to $I_j$. Since $I_j$ is maximal, $I_i+I_j\neq I_j$ if and only if $I_i\neq I_j$. If $I_i+I_j=I_j$ for all $i\in J$, then $\{\alpha_i+(I_i+I_j)\mid\ i\in J\}$ is just the part $\{\alpha_i+I_i\mid\ I_i=I_j\}$ of the original exact
covering system. This part can not be a covering system as there exists $I_i\ne I_j$ for some $1\le i\le k$. Now we can use Lemma 3 to the new exact covering system. Hence, $I_j$
must be repeated at least
$$\min\{G(\frac{I_j}{I_i+I_j})\mid I_i+I_j\neq I_j,\ i\in J\}$$
times. Since
\begin{align*}
\min\{G(\frac{I_j}{I_i+I_j})\mid\ I_i+I_j\neq I_j,\ i\in J\}
&=\min\{G(\frac{I_j}{I_i+I_j})\mid\ I_i\neq I_j,\ i\in J\}\\
&\ge\min\{G(\frac{I_j}{I_i+I_j})\mid\ I_i\neq I_j\},
\end{align*}
the proof is complete. \qed

\end{document}